\documentclass[final]{siamltex}
\usepackage{amssymb, amsmath}
\usepackage{float,epsfig}
\usepackage{algpseudocode}
\usepackage{algorithm}

\newcommand{\ba}{\begin{array}}
\newcommand{\ea}{\end{array}}

\newcommand{\nrm}[1]{|\!|\!| {#1} |\!|\!|}

\newcommand{\be}{\begin{equation}}
\newcommand{\ee}{\end{equation}}
\newcommand{\beano}{\begin{eqnarray*}}
\newcommand{\eeano}{\end{eqnarray*}}

\def\R{{\mathbb R}}
\def\C{{\mathbb C}}

\def\lam{\lambda}

\def\diag{\mathrm{diag}}

\title{Backward Error of Matrix Rational Function }

\author{ Namita Behera \thanks{Department of Mathematics, Sikkim University,
Sikkim-737102, INDIA,({\tt nbehera@cus.ac.in, niku.namita@gmail.com })}}

\begin{document}

\maketitle

\begin{abstract} We consider a minimal realization of a rational matrix functions $R(\lam)$ of the form $ R(\lam) = \sum_{j = 0}^{m} \lambda^{j}A_{j} +C(\lam E - A)^{-1}B =: P(\lam) +C(\lam E - A)^{-1}B, $ where $A_{i} \in \mathbb{C}^{n \times n}, $ and $A, E, C, B $ are constant matrices of appropriate dimensions. We perturb the polynomial part and either $C$ or $B$ from the realization part. We derive explicit computable expressions of backward errors of approximate eigenvalue of $R(\lam).$ We also determine minimal perturbations for which approximate eigenvalue are exact eigenvalue of the perturbed matrix rational functions.
\end{abstract}




\section{Introduction}In this paper we study about the perturbation analysis of rational eigenvalue problem. Given a rational matrix in realization form $R(\lam)$ the main purpose of this paper is to investigate backward error analysis of $R(\lam)$. Particularly, we consider  a minimal realization of $R(\lam)$ of the form \be \label{rlg}  R(\lam) = \sum_{j = 0}^{m} \lambda^{j}A_{j} +C(\lam E - A)^{-1}B =: P(\lam) +C(\lam E - A)^{-1}B, \ee where $A_{i} \in \mathbb{C}^{n \times n}, $ and $A, E, C, B $ are constant matrices of appropriate dimensions and perturb the polynomial part and either the matrix $C$ or $B$ from the realization part. Then we derive explicit computable expressions of backward errors of approximate
eigenvalue of $R(\lam).$

Rational eigenvalue problems arises in many applications such as calculations of quantum dots, free vibration of plates with elastically attached masses, vibrations of fluid-solid structures and in control theory, see~\cite{voss03, vossh06, BHMT, planchard89} and the references therein. For example, the rational matrix eigenvalue problem \cite{planchard89}
$$R(\lambda)x := -Ax+\lambda Bx+\sum_{j=1}^{k}\frac{\rho \lambda}{k_{j}-\lambda m_{j}}C_{j}x = 0 $$
where the matrices $A, B$ and $C_{j}$ are symmetric and positive (semi-) definite, and they are typically large and sparse arises vibrations of a tube bundle immersed in an inviscid compressible fluid are governed under some simplifying assumptions by an elliptic eigenvalue problem with non-local boundary conditions.

A similar problem \cite{volker2004} $$R(\lambda)x= -Kx + \lambda Mx + \lambda^{2} \sum \limits _{j=1}^{k} \frac{1}{\omega_{j}-\lambda}C_{j}x =0, $$ arises when a generalized linear eigenproblem is condensed exactly.

Considering  a realization of $R(\lam)$ given in (\ref{rlg}), it is shown in~\cite{bai11} that the eigenvalues and the eigenvectors of $R(\lam)$ can be computed by solving the generalized eigenvalue problem for the pencil
\begin{equation}
\label{compr} \mathcal{C}_{1}(\lam) := \lam \left[
                         \begin{array}{cccc|c}
                           A_{m} &  &  &  &  \\
                            & I_{n} &  &  &  \\
                            &  & \ddots &  &  \\
                            &  &  & I_{n} &  \\
                            \hline
                            &  &  &  & -E \\
                         \end{array}
                       \right] +
\left[ \begin{array}{cccc|c}
                  A_{m-1} & A_{m-2}& \cdots & A_{0}& C  \\
                  -I_{n} & 0 & \cdots & 0 &  \\
                   & \ddots &  & \vdots  & \\
                   & & -I_{n} &0& \\
                   \hline
                  &  &  & B & A \\
                \end{array}
              \right],
\end{equation} where the void entries represent zero entries. The pencil $\mathcal{C}(\lam)$  referred  to as  a {\em companion linearization} of $R(\lam)$ in~\cite{bai11}.

For computing zeros (eigenvalues) and poles of rational matrix, linearizations of rational matrix have been introduced recently in~\cite{rafinami1, behera} via matrix-fraction descriptions (MFD) of rational matrix. Let $ G(\lam) = N(\lam) D(\lam)^{-1}$ be a {\em right coprime} MFD of $G(\lam)$,  where $N(\lam)$ and $D(\lam)$ are matrix polynomials with $D(\lam)$ being regular. Then the {\em  zero structure} of $G(\lam)$ is the same as the {\em eigenstructure} of $N(\lam)$  and the {\em pole structure} of $G(\lam)$ is the same as the {\em eigenstructure} of $D(\lam)$, see~\cite{kailath}. Also $G(\lam)$ can be uniquely written as $ G(\lam) = P(\lam) + Q(\lam),$ where $P(\lam)$ is a matrix polynomial and $Q(\lam)$ is strictly proper~\cite{kailath}. We define $\deg(G) := \deg(P),$ the degree of the polynomial part of $G(\lam).$

\begin{definition}[Linearization, \cite{rafinami1}] \label{ratlin} Let $ G(\lam) $ be an $n\times n$  rational matrix function (regular or singular) and let $ G(\lam) = N(\lam) D(\lam)^{-1}$ be a {\em right coprime} MFD of $G(\lam)$. Set $ r := \deg(\det(D(\lam))), \, p :=\max( n, r)$ and $m := \deg(G(\lam)).$ If $m \geq 1$ then an $(mn+r)\times (mn+r)$ matrix pencil $\mathbb{L}(\lam)$ of the form \be \label{rlin} \mathbb{L}(\lam) :=
\left[\begin{array}{c|c} X - \lam Y & \mathcal{C} \\  \hline
                          \mathcal{B} & A-\lam E \\
                        \end{array} \right] \ee
is said to be a linearization of $G(\lam)$ provided that there are $(mn+r)\times (mn+r)$ unimodular matrix  polynomials $\mathcal{U}(\lam)$ and $\mathcal{V}(\lam),$ and $p\times p$ unimodular matrix polynomials $ Z(\lam)$ and $ W(\lam)$ such that $ \mathcal{U}(\lam) \diag(I_{s-(mn+r)}, \; \mathbb{L}(\lam)) \mathcal{V}(\lam) = \diag(I_{s-n}, \, N(\lam))$ and $  Z(\lam) \diag( I_{p-r}, \, A - \lam E) W(\lam) = \diag( I_{p-n}, \, D(\lam))$ for $ \lam \in \C,$ where $A-\lam E$ is an $r\times r$ pencil with $E$ being nonsingular and $ s :=\max( mn+r, 2n).$
\end{definition}

Thus the zeros and poles of $G(\lam)$ are the eigenvalues of $ \mathbb{L}(\lam)$ and $A-\lam E,$ respectively.

Backward perturbation analysis determines the smallest perturbation for which a computed solution is an exact solution of the perturbed problem.Backward perturbation analysis play an important role in the accuracy assessment of computed solutions of rational eigenvalue problems. Further,it also plays an important role in the selection of an optimum linearization of an rational eigenvalue problem. This assumes significance due to the fact that linearization is a standard approach to solving a rational eigenvalue problem(see \cite{bai11} and the references therein). The main purpose of this paper is to undertake a detailed backward perturbation analysis of approximate eigenelements of rational matrix functions.

In the present work we addressed the backward error of the REP. We perturbed only the polynomial part and either $C$ or $B$ matrix from realization part given in (\ref{rlg}).

The paper is organized as follows: Section~2 contains some basic definitions and results on matrix polynomial. In section $3$ we define the backward error $\eta_{R}$ of an approximate eigenvalue and eigentriple of $R$ for the rational eigenvalue problem $R(\lambda)$. We perturb the polynomial part and some of the rational part. Then we derive the explicit computable expressions of backward error. We also find out the minimal perturbations for which approximate eigenelements are exact eigenelements of the perturbed matrix rational functions. Finally, in section~4 we derive the backward error of companion linearization of rational eigenvalue problem and analyze the comparison with the original one.

{\bf Notation:\,} We consider the H$ \ddot{\textrm{o}}$lder $p$-norm on $\mathbb{C}^{n}$ defined by $\|x\|_{p} :=(\sum_{i=1}^{n}|x_{i}|^{p})^{1/p} $ for $1\leq p<\infty$, and $\|x\|_{\infty} := \textrm{max}_{1\leq i\leq n}|x_{i}|$. We denote the set of $n\times n$ matrices with complex entries by $\mathbb{C}^{n\times n}$. We consider the spectral norm$\|.\|_{2}$ and the Frobenius norm $\|.\|_{F}$ on $\mathbb{C}^{n\times n}$ given by $$ \|A\|_{2}:= \max_{\|x\|_{2}=1}\|Ax\|_{2}\indent \textrm{and} \indent \|A\|_{F}:= (\textrm{Trace}A^{*}A)^{1/2} .$$ We denote the largest nonzero singular value of $A \in \mathbb{C}^{m \times n}$ by $\sigma_{\max} $ and smallest nonzero singular value of  $A \in \mathbb{C}^{m \times n}$ by $\sigma_{\min} $ and pseudoinverse of $A \in \mathbb{C}^{m \times n}$ by $A^{+}$. Consider $R(\lambda)+\Delta R(\lambda) = \sum_{j=0}^{m}(A_{j}+\Delta A_{j})\lambda^{j} + W(C +\Delta C, A +\Delta A, E +\Delta E, B +\Delta B)$. Denote $\Lambda_{m} = (1, \lambda, \ldots ,\lambda^{m-1}, \lambda^{m})$.

\section{Preliminaries}

Consider the polynomial eigenvalue problem (PEP)
$P(\lambda)x = 0$ and $y^{*}P(\lambda) = 0$, where $$ P(\lambda) = \sum_{i= 0}^{m}\lambda^{i}A_{i}, \indent A_{i} \in \mathbb{C}^{n \times n}, \,\,\, A_{m} \neq 0$$ is a matrix polynomial of degree $m$. Here $x$ and $y$ are right and left eigenvectors corresponding to eigenvalue $\lambda$.  The pair $(\lambda, x)$ is referred to as an eigenpair and the triple $(\lambda, x, y)$ is referred to as an eigentriple. We will assume throughout that $P$ is regular, that is $\det P(\lambda) \neq 0$.

\begin{definition}
Let $\mathbb{L}_{m}(\mathbb{C}^{n \times n})$ be the vector space of $n \times n$ matrix polynomials of degree at most $m$. Let $\| .\|$ be a norm on $\mathbb{C}^{n \times n}$.  For $1\leq p\leq \infty, $ define $\nrm{.} : \mathbb{L}_{m}(\mathbb{C}^{n \times n}) \rightarrow \R$ by
\begin{equation}\label{mpn}
\nrm{P} := \|(\|A_{0}\|, \cdots, \|A_{m}\|)\|_{p}, \indent 1 \leq p\leq \infty,
\end{equation}
where $P(z) = \sum\limits_{i=0}^{m}z^{i}A_i$ and $\|.\|_{p}$ is the H$\ddot{\textrm{o}}$lder's $p$-norm. Then $\nrm{.}$ is a norm on $\mathbb{L}_{m}(\mathbb{C}^{n \times n})$
and $\langle ., .\rangle_m : \mathbb{L}_{m}(\mathbb{C}^{n \times n}) \times \mathbb{L}_{m}(\mathbb{C}^{n \times n}) \rightarrow \C$ given by  $$ \langle P_1, P_2\rangle_{m} :=\sum_{i=0}^{m}\langle A_i, B_i\rangle $$ is an innerproduct on $\mathbb{L}_{m}(\mathbb{C}^{n \times n})$, where $P_1(\lam) = \sum_{j=0}^{m} \lam^{i}A_i$ and $P_2(\lam) = \sum_{j=0}^{m} \lam^{i}B_i$. Then the dual norm  $\nrm{.}_{*}$ of $ \nrm{.}$ is given by $$ \nrm{Y}_{*} = \max\{ | \langle X, Y\rangle_{m}|: \nrm{X} = 1 \}. $$

Given a norm $\nrm{.}$ on $\mathbb{L}_{m}(\mathbb{C}^{n \times n})$, we define the normwise backward error  of an approximate eigenpair $(x, \lambda)$ of $P(\lambda)$, where $\lambda$ is finite, is defined by
\begin{equation*}
\eta_{p}(\lambda, P) := \min \{ \nrm{\Delta P} : \Delta P \in ,\mathbb{L}_{m}(\mathbb{C}^{n \times n}), \,\,\, \det(P(\lambda)+\Delta P(\lambda)) = 0 \}
\end{equation*}
where $\Delta P(\lambda) = \sum_{i= 0}^{m}\lambda^{i} \Delta A_{i}$.
\end{definition}
Then explicitly $\eta_{p}(\lambda, P)$ is given by 
\begin{equation}
\eta_{p}(\lambda, P) = \min_{\|x\| =1 }\left\{\frac{\|P(\lambda)x\|}{\|(1, \lambda, \ldots , \lambda^{m})\|_{q}} : x\in \mathbb{C}^{n}\right\} \leq \nrm{P}_{p},
\end{equation}
where $1/p+1/q = 1$. Particularly, for $2$-norm and frobenious norm the backward error is same and is given by $$\eta_{2}(\lambda, P) = \eta_{F}(\lambda, P) = \frac{\sigma_{\min}(P(\lambda))}{\|(1, \lambda, \ldots , \lambda^{m})\|_{q}} .$$ Also, the explicit formula of the backward error is obtained in \cite{NRF07} is given by
\begin{equation}
\eta(x, \lambda) =\frac{\|P(\lambda)x\|_{2}}{(\sum_{i = 0}^{m}|\lambda^{i}| \|A_{i}\|_{2})\|x\|_{2}}.
\end{equation}

\section{Backward Errors for Rational Eigenvalue Problem}
Consider the rational eigenvalue problem (REP) $R(\lambda)x = 0$ and $y^{*}R(\lambda) = 0$, where
\begin{equation}
R(\lambda) = \sum_{j = 0}^{m}A_{j}\lambda^{j}- C(A-\lambda E)^{-1}B := P(\lambda) + W(\lambda),
\end{equation}
where $W(\lam) = C(A-\lambda E)^{-1}B$, the size of $A$ and $E$ is $r\times r$, the size of $C$ is $n \times r$ and size of $B$ is $r \times n$. Here $x$ and $y$ are right and left eigenvectors corresponding to the eigenvalue $\lambda$. The pair $(\lambda, x)$ is referred to as an eigenpair and the triple $(\lambda, x, y)$ is referred to as an eigentriple. The standard way of solving this problem is to convert $R$ into a linear polynomial see \cite{bai11, rafinami1}
\begin{equation}
\mathcal{C}_{1}(\lambda)z = 0,
\end{equation} where $\mathcal{C}_{1}(\lambda) =\lambda X +Y, \,\, Y, X \in \mathbb{R}^{nm+r} $ given in (\ref{compr}) and
 $$z = \left[
 \begin{array}{c}
 \lambda^{m-1}x \\
  \lambda^{m-2}x \\
   \vdots \\
   x \\
   \hline
    y \\
 \end{array}
 \right]$$  and $y = -(A-\lambda E)^{-1}Bx$,
with the same spectrum as $R$ and solve the eigenproblem for $\mathcal{C}_{1}(\lambda)z = (\lambda X +Y)z = 0,$ where $z$ is the right eigenvector of $\mathcal{C}_{1}(\lambda)$, and $\mathcal{C}_{1}$ is the companion linearization of $R$ see,\cite{rafinami1}.

We now develop a general framework for perturbation analysis for rational eigenvalue problem (REP). We use the following notations throughout  this chapter as follows:

Let $P(\lambda)=\sum_{i=0}^{m}A_{i}\lambda^{i}$ is a matrix polynomial of degree $\leq m$, $A_{i}\in\, \mathbb{C}^{n \times n}, i=0 : m$,
$R(\lambda)= P(\lambda)+ C(A-\lambda E)^{-1}B$, where $C \in \, \mathbb{C}^{n\times r}, A, E \in \, \mathbb{C}^{r \times r}, B \in \, \mathbb{C}^{r \times n}$. Now perturb the coefficient of the matrix $R(\lambda)$ by $ \Delta A_{i}, i=0:m, \Delta C, \Delta A, \Delta E, \Delta B$. Now question is how one can develop a general framework for perturbation analysis of rational eigenvalue problem (REP) $?$ Now consider the space of rational matrix functions of degree $n$,
$\mathbb{X}= \left((\mathbb{C}^{n \times n})^{m+1}, \mathbb{C}^{n\times m},\mathbb{C}^{m\times m}, \mathbb{C}^{m\times m}, \mathbb{C}^{m\times n} \right). $ Let $R\in \mathbb{X}$,
$$R(\lambda)= P(\lambda)+ C(A- \lambda E)^{-1} B, $$ where $P(\lambda) = \sum_{i=0}^{m}A_{i}\lambda^{i}, \, A_{i}\in \mathbb{C}^{n \times n}, C \in \, \mathbb{C}^{n\times r}, A, E \in \, \mathbb{C}^{r \times r}, B \in \, \mathbb{C}^{r \times n}$. Now $\mathbb{X}$ is vector space under the componentwise addition and scalar multiplication. We assume throughout that the matrix rational function $R(\lambda)$ is regular, that is, $\det R(\lambda) \neq 0$, for some $\lambda \in\mathbb{C}$. The roots of $q_{i}(\lambda)$ are the poles of $R(\lambda)$. $R(\lambda)$ is not defined on these poles. Needless to mention that for perturbation analysis it is necessary to choose a norm so as to measure the magnitude of perturbations. Thus when $\mathbb{X}$ is equipped with a norm the resulting normed linear space provides a general framework for perturbation analysis of matrix rational functions in $\mathbb{X}$. So let $\nrm{.}$ be a norm on $\mathbb{X}$. Then the normed linear space $\big(\mathbb{X}, \nrm{.} \big)$ provides a natural framework for spectral perturbation analysis of matrix rational functions in  $\mathbb{X}$. We now employ this framework and analyze perturbation theory of matrix rational functions. So one may ask : Is there a natural norm on $\mathbb{X}$ that facilitates systematic analysis of perturbation analysis of matrix rational functions$?$ We will see that there are plenty of norms on the space of matrix rational functions $\mathbb{X}$. Let $\|.\|$ be a norm on $\mathbb{C}^{n \times n}$ . For $1\leq p\leq \infty$, we define $\nrm{.}: \mathbb{X} \rightarrow \mathbb{R}$ by
$$ \nrm{R}_{p}:= \left\| \left(\|A_{0}\|, \cdots, \|A_{m}\|, \|C\|, \|A\|, \|E\|, \|B\|\right)\right\|_{p}$$
is a norm, where $R(z):= \sum_{i=0}^{m}A_{i}z^{i} + C(A-z E)^{-1}B$ and $\|.\|_{p}$ is the H$\ddot{\textrm{o}}$lder's $p$-norm on $\mathbb{C}^{m+5}$.
We denote the space $\mathbb{X}$ when equipped with the norm $\nrm{.}_{p}$ by $\big(\mathbb{X}_{p}, \, \|.\|\big)$. Now consider $\mathbb{C}^{n \times n}$ and define $\langle. , . \rangle : \mathbb{C}^{n \times n}\times \mathbb{C}^{n \times n}$ by $\langle X, Y \rangle := \textrm{Trace} (Y^{*}X)$.
Then $\langle. , . \rangle$ defines an inner product on $\mathbb{C}^{n \times n}$ and $\| X\|_{F}:= \sqrt{\langle X, X\rangle}$ is the Frobenius norm on $\mathbb{C}^{n \times n}$. For a fixed $Y\in \, \mathbb{C}^{n \times n}$, the map $X\rightarrow \langle X, Y \rangle$ is a linear functional on
$\mathbb{C}^{n \times n}$ then (by the Riesz representation theorem ) there exits a unique $Z\in \, \mathbb{C}^{n \times n}$ such that $F(X)= \langle X, Z \rangle$. Now let $\|.\|$ be a norm on $\mathbb{C}^{n \times n}$. Then $\|.\|_{*}: \mathbb{C}^{n \times n}\rightarrow \mathbb{R}$ given by
$$\|Y\|_{*}= \textrm{sup}\{|\langle X, Y\rangle|: X\in\,\mathbb{C}^{n \times n}, \|X\|= 1\}$$
defines a norm and is referred to as the dual norm of $\|.\|$. This shows that , For $R_{1}, R_{2}\in \, \big(\mathbb{X}_{2}, \, \|.\|_{F} \big)$,
$$ \langle R_{1}, \,  R_{2}\rangle_{\mathbb{X}}: = \sum_{i=0}^{m}\big \langle A_{i}, A_{i}'\big \rangle + \langle C_{1}, C_{2}\rangle + \langle A_{1}, A_{2}\rangle +\langle E_{1}, E_{2} \rangle + \langle B_{1}, B_{2}\rangle, $$ where $R_{1}(z):= \sum_{i=0}^{m}A_{i}z^{i} + C_{1}(A_{1}-z E_{1})^{-1}B_{1}$  and $R_{2}(z):= \sum_{i=0}^{m}A_{i}'z^{i} + C_{2}(A_{2}-z E_{2})^{-1}B_{2}$ \\
defines an inner product on
$\big(\mathbb{X}_{2}, \, \|.\|_{F} \big)$.Thus $\big(\mathbb{X}_{2}, \, \|.\|_{F} \big)$ is a Hilbert space and $\nrm{R}_{2}= \sqrt{\langle R, R \rangle}_{\mathbb{X}}$. Let $\nrm{.}$ be a norm on $ \mathbb{X}$ and $\nrm{.}_{*}$ the dual norm $$\nrm{Y}_{*} = \sup\{ |\langle X, Y\rangle | : \nrm{X} = 1\}. $$ In particular for $p$-norm the dual is given by
$$\nrm{R}_{*}:= \left\| \left(\|A_{0}\|_{*}, \cdots, \|A_{m}\|_{*}, \|C\|_{*}, \|A\|_{*}, \|E\|_{*}, \|B\|_{*}\right)\right\|_{q}. $$

%
%
%

Now, we derive the backward error of approximate eigenelements of rational matrix function where we  perturbed the polynomial part and the matrix either $C$ or $B$ matrix from realization part given in (\ref{rlg}). Corresponding to this we define the backward error is as follows:

\begin{definition}
The normwise backward error  of an approximate eigenelement $ \lam$ of $R(\lambda)$, where $\lambda$ is finite, is defined by
$$\eta_{p}(\lam, R) := \min \{ \|[\Delta P \,\, \Delta B \,\, \Delta C]\| : \Delta R \in \mathbb{X}, \det (R(\lambda)+\Delta R(\lambda))= 0 \} $$ $$\eta_{p}(\lambda, R) := \min \{ \|[\Delta P \,\, \Delta B \,\, \Delta C]^{T}\| : \Delta R \in \mathbb{X}, \det (R(\lambda)+\Delta R(\lambda))= 0 \}$$
$$
\eta_{p}(\lambda, R) := \min \{ \nrm{\Delta R} : \Delta R \in \mathbb{X}, \det (R(\lambda) + \Delta R(\lambda))= 0 \}.
$$
\end{definition}

\begin{lemma}
Let $R(\lambda) = A_{0} + \lambda A_{1}+ C(A-\lambda E)^{-1} B$, $\Delta R(\lambda) = \Delta A_{0}+\lambda \Delta A_{1}+ \Delta C(A - \lambda E)^{-1}B$ and $\lambda \in \mathbb{C}$. Assume that $R(\lambda)$ is nonsingular. Set $T(\lambda) = \left[
                                                                      \begin{array}{c}
                                                                        R^{-1}(\lambda) \\
                                                                        \lambda R^{-1}(\lambda) \\
                                                                         W_{1}(\lambda) R^{-1}(\lambda) \\
                                                                      \end{array}
                                                                    \right]$, $W_{1}(\lambda)= (A-\lambda E)^{-1}B$,
$\Delta = \left[
            \begin{array}{ccc}
              \Delta A_{0} & \Delta A_{1} & \Delta C \\
            \end{array}
          \right]
$ and $v = R(\lambda)x$. Then the following statements are equivalent.

\begin{itemize}

\item [(i)] $\det(R(\lambda)+\Delta R(\lambda)) = 0$

\item [(ii)] $\Delta T(\lambda) v = -v$.
\end{itemize}
\end{lemma}

\begin{proof}Let $\Delta R(\lambda) = \Delta A_{0}+\lambda \Delta A_{1}+ \Delta C(A-\lambda E)^{-1}B$. Set $W_{1}(\lambda)= (A-\lambda E)^{-1}B$. Then $\Delta R(\lambda) = \Delta A_{0}+\lambda \Delta A_{1}+\Delta C W_{1}(\lambda)$. Now $\det (R(\lambda) + \Delta R(\lambda))= 0$. That means there exists a nonzero vector $x$ such that $\|x\| = 1$ and $R(\lambda)x + \Delta R(\lambda)x = 0$. So we have $[I + \Delta R(\lambda)R^{-1}(\lambda)]R(\lambda)x = 0$. Put $v = R(\lambda)x, $ implies that $x = R^{-1}(\lambda)v$. Then $\Delta R(\lambda)R^{-1}(\lambda)v = -v$. Hence the result follows.
\end{proof}

\begin{corollary}
Let $R$ satisfies all the given conditions of Lemma $3.2$. Then $$\eta(\lambda, v, R) = \min\{\|\Delta\| : \Delta T(\lambda) v = -v\}. $$
\end{corollary}

\begin{theorem}
Consider the subordinate matrix norm $\|.\|$ on $\mathbb{C}^{n \times n}$. Let $R(\lambda) = A_{0}+ \lambda A_{1} + C(A-\lambda E)^{-1} B$ and $\Delta R(\lambda) = \Delta A_{0} + \lambda \Delta A_{1}+ \Delta C(A - \lambda E)^{-1}B$. Set
$\Delta = \left[
            \begin{array}{ccc}
              \Delta A_{0} & \Delta A_{1} & \Delta C \\
            \end{array}
          \right]
$. Chose $x = R(\lambda)^{-1}v$. Set $T(\lambda) = \left[
                                                    \begin{array}{c}
                                                       R^{-1}(\lambda)  \\
                                                      \lambda R^{-1}(\lambda)  \\
                                                       W_{1}(\lambda) R^{-1}(\lambda) \\
                                                         \end{array}
                                                        \right]
 $ and $W_{1}(\lam)= (A-\lam E)^{-1}B$. Then we have $$ \eta(\lam, R) = \min_{\|v\| = 1}\left\{\frac{v (T(\lambda)v)^{*} }{\|T(\lam)v\|^{2}}\right\}. $$
Then we have $$ \eta(\lambda, R) = \min_{\|v\| = 1}\left\{\frac{1}{\|T(\lambda)v\|}\right\}. $$
\end{theorem}

\begin{proof}
Let $\Delta R(\lambda) = \Delta A_{0} +\lambda \Delta A_{1}+ \Delta C(A-\lambda E)^{-1}B$. Set $W_{1}(\lambda)= (A-\lambda E)^{-1}B$. Then $\Delta R(\lambda) = \Delta A_{0}+\lambda \Delta A_{1}+ \Delta C W_{1}(\lambda)$. Now $\det (R(\lambda)+\Delta R(\lambda))= 0$. That means there exists a nonzero vector $x$ such that $\|x\| = 1$ and $R(\lambda)x + \Delta R(\lambda)x = 0$. So we have $[I +\Delta R(\lambda)R^{-1}(\lambda)]R(\lambda)x = 0$. Put $v = R(\lambda)x, $ implies that $x = R^{-1}(\lambda)v$. Then $\Delta R(\lambda)R^{-1}(\lambda)v = -v$. $\Rightarrow \left[
                                                                           \begin{array}{ccc}
                                                                            \Delta A_{0} & \Delta A_{1} & \Delta C \\
                                                                              \end{array}
                                                                                  \right]
T(\lambda)v = -v$.
Define
$$\Delta A_{0} := \frac{-v[R^{-1}(\lambda)v]^{*}}{\sigma_{\max}(T(\lambda))}, \,\,\, \Delta A_{1} := \frac{-\bar{\lambda}v[R^{-1}(\lambda)v]^{*}}{\sigma_{\max}(T(\lambda))}, \mbox{   and   } \Delta C := \frac{-v[W(\lambda)R^{-1}(\lambda)v]^{*}}{\sigma_{\max}(T(\lambda))}.$$  Now consider the matrix rational function with perturbing only the polynomial part and the matrix $C$ only. Consider $ \Delta R(\lambda) := \Delta A_{0}+\lambda\Delta A_{1}+\Delta C(A-\lambda E)^{-1}B$. Then by the construction, we get $R(\lambda)v+\Delta R(\lambda)v = 0$ and $\|\Delta\| = \frac{1}{\|T(\lambda)v\|}.$
\end{proof}

\begin{theorem}
Let $R(\lambda) = A_{0}+\lambda A_{1}+ C(A-\lambda E)^{-1} B$ be regular. Chose $x = R(\lambda)^{-1}v$. Set
$\Delta = \left[
            \begin{array}{ccc}
              \Delta A_{0} & \Delta A_{1} & \Delta C \\
            \end{array}
          \right]
$, $T(\lambda) = \left[
                  \begin{array}{c}
                   R^{-1}(\lambda) \\
                    \lambda R^{-1}(\lambda)  \\
                     W_{1}(\lambda) R^{-1}(\lambda) \\
                      \end{array}
                       \right]
 $ and $W_{1}(\lambda)= (A-\lambda E)^{-1}B$. Then for Frobenious norm $\|.\|_{F}$ and $2$-norm on $\mathbb{C}^{n \times n}$ we have
$$ \eta(\lambda, R) = \frac{1}{\sigma_{\max}(T(\lambda))}.$$
\end{theorem}

\begin{proof}
Let $v$ be unit right singular vector of $T(\lambda)$. Let $\Delta R(\lambda) = \Delta A_{0}+\lambda \Delta A_{1}+\Delta C(A-\lambda E)^{-1}B$. Set $W_{1}(\lambda)= (A-\lambda E)^{-1}B$. Then $\Delta R(\lambda) = \Delta A_{0}+\lambda \Delta A_{1}+\Delta C W_{1}(\lambda)$. Now $\det (R(\lambda)+ \Delta R(\lambda))= 0$. That means there exists a nonzero vector $x$ such that $\|x\| = 1$ and $R(\lambda)x+\Delta R(\lambda)x = 0$. So we have $[I + \Delta R
(\lambda)R^{-1}(\lambda)]R(\lambda)x = 0$. Put $v = R(\lambda)x, $ implies that $x = R^{-1}(\lambda)v$. Then $\Delta R(\lambda)R^{-1}(\lambda)v = -v$.
Define
$$\Delta A_{0} := \frac{-v[R^{-1}(\lambda)v]^{*}}{\sigma_{\max}(T(\lambda))}, \,\,\, \Delta A_{1} := \frac{-\bar{\lambda}v[R^{-1}(\lambda)v]^{*}}{\sigma_{\max}(T(\lambda))}, \mbox{   and   } \Delta C := \frac{-v[W(\lambda)R^{-1}(\lambda)v]^{*}}{\sigma_{\max}(T(\lambda))}.$$  Now consider the matrix rational function with perturbing only the polynomial part and the matrix $L$ only. Consider $ \Delta R(\lambda) := \Delta A_{0}+\lambda\Delta A_{1} +\Delta C(A-\lambda E)^{-1}B$. Then by the construction, we get $R(\lambda)x + \Delta R(\lambda)x = 0$. Since each $\Delta A_{0}, \Delta A_{1} \mbox{ and }\Delta C$ is a rank $1$ matrix, the spectral and the Frobenius norms of $\Delta A_{0} ,  \Delta A_{1} \mbox{ and }\Delta C$ are same. Consequently, $\nrm{\Delta R}_{p}$ is same for the spectral and the Frobenius norms on $\mathbb{C}^{n \times n}$. Now, we get $\nrm{\Delta R}_{2} = \frac{1}{\sigma_{\max}(T(\lambda))}$.
Hence the result follows.
\end{proof}

\begin{theorem}
Let $R(\lambda) = \sum_{i =0}^{m}\lambda^{i}A_{i} + C(A-\lambda E)^{-1} B$ be regular and $\Delta R(\lambda) = \sum_{i =0}^{m}\lambda^{i}\Delta A_{i} + \Delta C(A-\lambda E)^{-1} B$. Chose $x = R(\lambda)^{-1}v$. Set $T(\lambda) = \left[
                                                                      \begin{array}{c}
                                                                       \Lambda_{m}^{T}\otimes R^{-1}(\lambda) \\
                                                                        W_{1}(\lambda) R^{-1}(\lambda) \\
                                                                      \end{array}
                                                                    \right],$
$W_{1}(\lambda)= (A-\lambda E)^{-1}B$, and $\Delta = \left[
            \begin{array}{cccc}
              \Delta A_{0} & \cdots & \Delta A_{m} & \Delta C \\
            \end{array}
          \right]
$. Then for Frobenious norm $\|.\|_{F}$ and $2$-norm on $\mathbb{C}^{n \times n}$ we have
$$ \eta_{2}(\lambda, R) = \eta_{F}(\lambda, R) = \frac{1}{\sigma_{\max}(T(\lambda))}$$
\end{theorem}

\begin{proof}
Define $$\Delta A_{i} = \frac{- \mbox {sign}(\lambda^{i})|\lambda|^{i}v[R^{-1}(\lambda)v]^{*}}{\sigma_{\max}(T(\lambda))}, \indent i =0:m , \,\,\, \Delta C:= \frac{-v[W(\lambda)R^{-1}(\lambda)v]^{*}}{\sigma_{\max}(T(\lambda))}$$
Then by the construction, we get $[R(\lambda) + \Delta R(\lambda)]x = 0$ and
\end{proof}

\begin{lemma}
Let $R(\lambda) = A_{0}+\lambda A_{1} + C(A-\lambda E)^{-1} B$, $\Delta R(\lambda) = \Delta A_{0}+\lambda \Delta A_{1}+ C(A-\lambda E)^{-1}\Delta B$, and $\lambda \in \mathbb{C}$. Assume that $R(\lambda)$ is nonsingular. Set $S(\lambda) = \left[
                                                                      \begin{array}{ccc}
                                                                        R^{-1}(\lambda) & \lambda R^{-1}(\lambda) & R^{-1}(\lambda)\widehat{W}(\lambda) \\
                                                                      \end{array}
                                                                    \right]$, $\widehat{W}(\lambda) = C(A-\lambda E)^{-1}$
and $\Delta = \left[
            \begin{array}{c}
              \Delta A_{0} \\
              \Delta A_{1} \\
               \Delta B \\
            \end{array}
          \right]
$. Then the following statements are equivalent.
\begin{itemize}

\item [(i)] $\det(R(\lambda) + \Delta R(\lambda)) = 0$

\item [(ii)] $S(\lambda)\Delta  x = -x$.
\end{itemize}
\end{lemma}

\begin{proof}
Let $\Delta R(\lambda) = \Delta A_{0}+\lambda \Delta A_{1}+ C(A-\lambda E)^{-1}\Delta B$. Set $\widehat{W}(\lambda)= C(A-\lambda E)^{-1}$. Then $\Delta R(\lambda) = \Delta A_{0}+\lambda \Delta A_{1}+ \widehat{W}(\lambda)\Delta B $. Now $\det (R(\lambda) +\Delta R(\lambda))= 0$. That means there exists a nonzero vector $x$ such that $\|x\| = 1$ and $R(\lambda)x + \Delta R(\lambda)x = 0$. So we have $[I + R^{-1}(\lambda)\Delta R(\lambda)]x = 0$. Then $R^{-1}(\lambda)\Delta R(\lambda)x = -x$. Hence the result follows.
\end{proof}

\begin{corollary}
Let $R$ satisfies all the given conditions of Lemma $3.7$. Then $$\eta(\lambda, x, R) = \min\{\|\Delta\| :  S(\lambda)\Delta x = -x\} $$
\end{corollary}

\begin{theorem}
Consider the subordinate matrix norm $\|.\|$ on $\mathbb{C}^{n \times n}$. Let $R(\lambda) = A_{0} +\lambda A_{1} + C(A-\lambda E)^{-1} B$ and $\Delta R(\lambda) = \Delta A_{0} +\lambda \Delta A_{1} + C(A-\lambda E)^{-1}\Delta B$. Set
$\Delta = \left[
            \begin{array}{c}
              \Delta A_{0} \\
              \Delta A_{1} \\
               \Delta B \\
            \end{array}
          \right]
$ and $S(\lambda) = \left[
                        \begin{array}{ccc}
                          R^{-1}(\lambda) & \lambda R^{-1}(\lambda) & R^{-1}(\lambda) \widehat{W}(\lambda) \\
                                                                      \end{array}
                                                                    \right]
 $. If $S(\lambda)S(\lambda)^{+}x = x$.
Then we have $$ \eta(\lambda, R) = \min_{\|x\| = 1}\left\{\|[S(\lambda)]^{+}xx^{*}\|\right\}. $$
\end{theorem}

\begin{theorem}
Consider the space $(\mathbb{X}, \|.\|)$ corresponding to subordinate matrix norm $\|.\|$ on $\mathbb{C}^{n \times n}$. Let $R  \in (\mathbb{X}, \|.\|)$ is given by $R(\lambda) = A_{1}+\lambda A_{2}+ C(A-\lambda E)^{-1} B$. Perturbing only the pencil part and the matrix $B$. Set \\
$S(\lambda) = \left[
 \begin{array}{ccc}
  R^{-1}(\lambda) & \lambda R^{-1}(\lambda) & R^{-1}(\lambda) W(\lambda) \\
  \end{array}
   \right].
 $ If $S(\lambda)S(\lambda)^{+}v = v$. Then we have $$ \eta_{p}(\lambda, R) = \min_{\|x\| = 1}\left\{\frac{1}{\|[S(\lambda)]^{+}xx^{*}\|}\right\}, $$ where $1/p+1/q = 1$.
\end{theorem}

\begin{proof}
Let $\widehat{W}(\lambda) = C(A-\lambda E)^{-1}$ and $\Delta R(\lambda) = \Delta A_{0}+\lambda\Delta A_{1}+ \widehat{W}(\lambda) \Delta B$. Now $\det (R(\lambda) + \Delta R(\lambda)) = 0$. Then there exists $x \neq 0$ with $\|x\| = 1$ such that $[R(\lambda) + \Delta R(\lambda)]x = 0$. Then $[I + R^{-1}(\lambda)\Delta R(\lambda)]x = 0$. Define $Z(\lambda) = [S(\lambda)S(\lambda)^{*}]^{-1}$ and
$$ \Delta A_{0} = R^{-1}(\lambda)^{*}Z(\lambda)xx^{*} \indent \Delta A_{1} = \bar{\lambda}R^{-1}(\lambda)^{*}Z(\lambda)xx^{*} \indent \Delta B =
(R^{-1}(\lambda)\widehat{W}(\lambda))^{*}Z(\lambda)xx^{*}. $$
Then by the construction, we get $[I + R^{-1}(\lambda)\Delta R(\lambda)]x = 0$ and $\| \Delta\| = \frac{1}{\|[S(\lambda)]^{+}xx^{*}\|}. $
\end{proof}

\begin{theorem}
Let $R(\lambda) = A_{0} +\lambda A_{1} + C(A-\lambda E)^{-1} B$ be regular, $\Delta R(\lambda) = \Delta A_{0}+\lambda \Delta A_{1} + C(A-\lambda E)^{-1}\Delta B$ and $\lambda\in\mathbb{C}$. Set $\Delta = \left[
            \begin{array}{c}
              \Delta A_{0} \\
              \Delta A_{1} \\
               \Delta B \\
            \end{array}
          \right]
$, \\ $S(\lambda) = \left[
                 \begin{array}{ccc}
                 R^{-1}(\lambda) & \lambda R^{-1}(\lambda) & R^{-1}(\lambda)\widehat{W}(\lambda) \\
                 \end{array}
                  \right]
 $ and $\widehat{W}(\lambda) = C(A-\lambda E)^{-1}$. If $S(\lambda)S(\lambda)^{+}v = v$. Then for Frobenious norm $\|.\|_{F}$ and $2$-norm on $\mathbb{C}^{n \times n}$ we have
$$ \eta_{2}(\lambda, R) = \eta_{F}(\lambda, R) = \frac{1}{\sigma_{\min}(S(\lambda))}. $$
\end{theorem}


\begin{proof}
Let $v$ be the right singular vector of $S(\lambda)$. Let $\Delta R(\lambda) = \Delta A_{0} + \lambda\Delta A_{1}+ W(\lambda) \Delta B$. Set $\widehat{W}(\lambda) = C(A-\lambda E)^{-1}$. Then $\Delta R(\lambda) = \Delta A_{0} + \lambda \Delta A_{1}+ \widehat{W}(\lambda)\Delta B$. Define $Z(\lambda) = [S(\lambda)S(\lambda)^{*}]^{-1}$ and
$$ \Delta A_{0} = -R^{-1}(\lambda)^{*}Z(\lambda)vv^{*} \,\, \Delta A_{1} = -\bar{\lambda}R^{-1}(\lambda)^{*}Z(\lambda)vv^{*} \,\, \Delta B = -(R^{-1}(\lambda)\widehat{W}(\lambda))^{*}Z(\lambda)vv^{*}. $$
Then by the construction, we get $[I+ R^{-1}(\lambda)\Delta R(\lambda)]v = 0$. Since each $\Delta A_{0}, \Delta A_{1} \mbox{ and }\Delta B$ is a rank $1$ matrix, the spectral and the Frobenius norms of $\Delta A_{0} , \Delta A_{1} \mbox{ and }\Delta B$ are same. Consequently, $\|\Delta R\|_{2}$ is same for the spectral and the Frobenius norms on $\mathbb{C}^{n \times n}$  and $\| \Delta\|_{2} = \frac{1}{\sigma_{\min}S(\lambda)}$. Hence the result follows.
\end{proof}

\begin{theorem}
Let $R(\lambda) = \sum_{i =0}^{m}\lambda^{i}A_{i} + C(A-\lambda E)^{-1} B$ be regular and $\Delta R(\lambda) = \sum_{i =0}^{m}\lambda^{i}\Delta A_{i} + C(A -\lambda E)^{-1}\Delta B$. Set $\Delta = \left[
            \begin{array}{c}
              \Delta A_{0} \\
              \vdots \\
              \Delta A_{m} \\
               \Delta B \\
            \end{array}
          \right]$, $S(\lambda) = \left[
                                                                      \begin{array}{cc}
                                                                       \Lambda_{m}\otimes R^{-1}(\lambda) &
                                                                         R^{-1}(\lambda)\widehat{W}(\lambda) \\
                                                                      \end{array}
                                                                    \right]
 $ and $\widehat{W}(\lambda) = C(A-\lambda E)^{-1}$. If $S(\lambda)S(\lambda)^{+}v = v$. Then for Frobenious norm $\|.\|_{F}$ and matrix $2$-norm on $\mathbb{C}^{n \times n}$ we have
$$ \eta_{2}(\lambda, R) = \eta_{F}(\lambda, R) = \frac{1}{\sigma_{\min}(S(\lambda))}. $$
\end{theorem}

\begin{theorem}\label{bet}
Let $R$ be regular. Perturbing only polynomial part and keep the rational part as it is. Then for matrix $2$ norm we have
$$ \eta(\lambda, R) \leq \frac{\sigma_{\min}(R(\lambda))}{\|( 1, \lambda, \cdots , \lambda^{m} )\|_{2}}. $$
\end{theorem}

\section{Companion Linearization}
Let $\mathcal{C}_{1}(\lambda) = \lambda \mathcal{X} + \mathcal{Y}$ be the first companion linearization of $R(\lam)$. Let $\|.\|$ be a norm on $\mathbb{C}^{n\times n}$. For $1 \leq p\leq \infty$. $\nrm{\mathcal{C}_{1}} = \| (\|\mathcal{Y} \|, \|\mathcal{X}\|)\|_{p}$. Therefore by H$\ddot{\textrm{o}}$lder's inequality we have $$ \|\mathcal{C}_{1}(\lambda)\| \leq \nrm{\mathcal{C}_{1}}\|(1, \lambda)\|_{q}. $$

\begin{proposition}
Let $\|.\|$ be the subordinate matrix norm on $\mathbb{C}^{n \times n}$ and $\lambda \in \mathbb{C}$. Then $$ \eta_{p}(\lambda, \mathcal{C}_{1}) = \min_{\|z\| = 1} \left\{\frac{\|\mathcal{C}_{1}(\lambda)z\|}{\|(1, \lambda)\|_{q}} : z\in \mathbb{C}^{nm + r}\right\} = \left(\|(1, \lambda)\|_{q}\|\mathcal{C}_{1}(\lambda)^{-1}\|\right)^{-1} \leq \nrm{\mathcal{C}_{1}}_{q}, $$ where $1/p+1/q = 1$. For matrix $2$-norm and Frobenius norm we get $$ \eta_{2}(\lam, \mathcal{C}_{1}) = \frac{\sigma_{\min}(\mathcal{C}_{1}(\lambda))}{\|(1, \lambda)\|_{2}}. $$
\end{proposition}

\begin{theorem}
 Let $\eta_{2}(\lambda, R)$ is given in Theorem (\ref{bet}). Then
\begin{align*}
 \frac{\eta_{2}(\lambda, \mathcal{C}_{1})}{\eta_{2}(\lambda, R)} \geq \frac{\sqrt{m}}{\sqrt{2}}\frac{\sigma_{\min}(\mathcal{C}_{1}(\lambda))}{\sigma_{\min}(R(\lambda))}.
\end{align*}
\end{theorem}

\begin{proof}
Let $\Lambda_{m} = (1, \lambda, \ldots , \lambda^{m})$.
\begin{align*}
\frac{\eta_{2}(\lam, \mathcal{C}_{1})}{\eta_{2}(\lam, R)}  \geq \frac{\sigma_{\min}(\mathcal{C}_{1}(\lam))}{\sigma_{\min}(R(\lam))} \frac{\|(1, \lam, \ldots, \lam^{m})\|_{2}}{\|(1, \lam)\|_{2}} \geq \frac{\sigma_{\min}(\mathcal{C}_{1}(\lam))}{\sigma_{\min}(R(\lam))}\frac{\|\Lambda_{m}\|_{2}}{\|(1, \lambda)\|_{2}}.
\end{align*}
We know that $\frac{1}{\sqrt{2}} \leq \frac{\|\Lambda_{m}\|_{2}}{\|\Lambda_{m-1}\|_{2} \|(1, \lam)\|_{2}} \leq 1$. Now
\begin{align*}
\frac{\|\Lambda_{m}\|_{2}}{\|(1, \lambda)\|_{2}} = \frac{\|\Lambda_{m}\|_{2}\|\Lambda_{m-1}\|_{2}}{\|\Lambda_{m-1}\|_{2} \|(1, \lam)\|_{2}} \geq\frac{1}{\sqrt{2}}\|\Lambda_{m-1}\|_{2} = \frac{1}{\sqrt{2}}(1+|\lambda|^{2}+ \ldots + |\lambda|^{2(m-1)})^{1/2}
\end{align*}
If $|\lambda| \geq 1$ then $\frac{\|\Lambda_{m}\|_{2}}{\|(1, \lambda)\|_{2}}\geq \frac{1}{\sqrt{2}}\sqrt{m}$. So
$$\frac{\eta_{2}(\lam, \mathcal{C}_{1})}{\eta_{2}(\lam, R)} \geq \frac{\sqrt{m}}{\sqrt{2}} \frac{\sigma_{\min}(\mathcal{C}_{1}(\lam))}{\sigma_{\min}(R(\lam))}. $$
\end{proof}

{\bf Conclusion}
We defined the backward error $\eta_{R}$ of an approximate eigenvalue and eigentriple of $R$  for the rational eigenvalue problem $R(\lambda)$ given in (\ref{rlg}). Then we derived the explicit computable expressions of backward error of approximate eigen-
value of $R(\lambda)$. We also found out the minimal perturbations for which approximate eigenelements are exact eigenelements of the perturbed matrix rational functions. Finally, we derived the backward error of companion linearization of rational eigenvalue problem and analyzed the comparison with the original one.

\end{document}